\documentclass[12pt,a4paper]{amsart}
 
\usepackage{amsmath,amssymb,amsthm} %,bbm} 
\usepackage{enumitem}  
\usepackage[overload]{textcase}
 
\usepackage{epsfig,psfrag,color} 

\usepackage{tikz}
\usetikzlibrary{shapes}

\usepackage[colorlinks,bookmarks,linkcolor=black,citecolor=green]{hyperref} 
\setlist{itemsep=1pt,parsep=0pt,topsep=2pt,partopsep=0pt}  
\setenumerate{leftmargin=*} 
 
\def\endofFact{\hfill\scalebox{.6}{$\Box$}}

\let\subset\subseteq  
\let\eps\varepsilon 
\let\phi\varphi
\let\rho\varrho 
\newtheorem{theorem}{Theorem}%[section] 
\newtheorem{lemma}[theorem] {Lemma}

\newtheorem{definition}[theorem] {Definition}  
\theoremstyle{remark}

\newcommand{\NATS}{\mathbb{N}}

\DeclareMathOperator{\pl}{pl} % planarity  
\DeclareMathOperator{\dcup}{\ensuremath{\mathaccent\cdot \cup}} % disjoint union
\newcommand{\oldqed}{}

\title[Maximum planar subgraphs in dense graphs]{Maximum planar subgraphs in
dense graphs}

  \author[Peter Allen]{Peter Allen}
  \address{Department of Mathematics, London School of Economics, Houghton Street, London WC2A 2AE, United Kingdom.}
  \email{\{p.d.allen|j.skokan\}@lse.ac.uk}

  \author[Jozef Skokan]{Jozef Skokan}

  \author[Andreas W\"urfl]{Andreas W\"urfl}
  \address{Zentrum Mathematik, Technische Universität M\"unchen, Boltzmann\-strasse~3 D-85748, Garching bei M\"unchen, Germany.}   
  \email{wuerfl@ma.tum.de}

\date{\today} 

\begin{document} 
 
\begin{abstract} 
 K\"uhn, Osthus and Taraz showed that for each $\gamma>0$ there exists 
 $C$ such that any $n$-vertex graph with minimum degree $\gamma n$ 
 contains a planar subgraph with at least $2n-C$ edges. We find the 
 optimum value of $C$ for all $\gamma< 1/2$ and sufficiently large $n$.
\end{abstract}  

\maketitle

\section{Introduction} %\label{sec:Intro}

\thispagestyle{empty}

A way to reformulate typical questions in extremal graph theory is the 
following. Given a property $\mathcal{P}$ and an edge density (or 
minimum vertex degree, etc.), what is the `largest' member of 
$\mathcal{P}$ which must be contained in an $n$-vertex graph $G$ with 
the given density? For many problems in extremal graph theory, the
property $\mathcal{P}$ is somewhat trivial (for example, in Tur\'an's 
theorem, $\mathcal{P}$ is the set of cliques). However this is not always 
the case: for example, in the Erd\H{o}s-Stone~\cite{ErdSto46} theorem, 
$\mathcal{P}$ is the set of complete $r$-partite graphs, and the problem
of determining the `largest' complete $r$-partite subgraph remains 
active, with most recently results and generalisations due to Nikiforov 
\cite{Niki11}. In 2005, K\"uhn, Osthus and Taraz~\cite{KuOsTa05} 
suggested the study of the property $\mathcal{P}$ consisting of all 
planar graphs, which, while well-studied in other parts of graph theory,
have received relatively little attention from extremal graph theorists.
\smallskip

A plane graph is a drawing of a graph in the plane with no crossing edges. A
graph is called planar if it has a plane graph drawing. The planarity of a graph
$G$ is defined as the maximum number of edges in a planar subgraph of $G$.
We denote the planarity of $G$ by $\pl(G)$. K\"uhn, Osthus and 
Taraz~\cite{KuOsTa05} investigated the connection between the minimum 
degree $\delta(G)$ and planarity $\pl(G)$ of a graph $G$ by studying the
parameter
$$
 \pl(n,d):=\min\{\pl(G): |G|=n, \delta(G)\ge d\}.
$$
Among other results they proved the following theorem.

\begin{theorem} \label{thm:KOTtriang}
For each $\gamma>0$ there exists a constant $n_\gamma$ such that $\pl\big(n, (2/3+\gamma) n\big) = 3n-6$ for every integer $n\ge n_\gamma$.
\end{theorem}

This was later improved by K\"uhn and Osthus~\cite{KuhOst05} to the following
result with the optimal bound on the minimum degree.

\begin{theorem} \label{thm:KOtriang}
There exists $n_2$ such that $\pl(n, 2n/3) = 3n-6$ for every integer $n\ge n_2$.
\end{theorem}

More recently, Cooley, {\L}uczak, Taraz and W\"urfl~\cite{CLTW12} showed 
the following threshold behaviour of $\pl(n,d)$ at minimum degree $d=n/2$.

\begin{theorem} \label{thm:CLTWjump}
For every $\mu>0$ there exists $n_\mu$ such that, for every $n\ge n_\mu$, we have that
\begin{align*}
 \pl\big(n, \lceil n/2\rceil\big) &\ge (2.25-\mu)n & &\text{ for $n$ odd,}\\
\intertext{and} 
 \pl(n,n/2+1) &\ge (2.5-\mu)n & &\text{ for $n$ even.}
\end{align*}
\end{theorem}

This indeed constitutes a threshold behaviour since $\pl(n,\lfloor n/2\rfloor)\le 2n-4$ 
for all integers $n$ as one can see from the class of complete bipartite graphs. For 
smaller values of $d$ one does not observe such rapid changes in the planarity. 
Indeed, K\"uhn, Osthus and Taraz~\cite{KuOsTa05} showed that $\pl(n,d)$ varies 
only by a constant term for the whole range of $d=\gamma n$ with $\gamma \in (0,1/2)$.

%\AW{Is this too much results of other people? If you think so, I suggest to omit Thm~\ref{thm:CLTWjump}.}

\begin{theorem}\label{thm:KOTquad}
 For each $\gamma>0$ there is $C=C(\gamma)$ such that 
 $\pl(n,\gamma n) \ge 2n - C$ for every integer $n$.
\end{theorem}

For $\gamma<1/2$ this is optimal up to the value of the constant $C$.
For $\gamma\ge 1/2$ the above statement trivially holds: a Hamilton cycle with chords 
from one vertex on the inner face and from another vertex on the outer face proves 
that every $n$-vertex graph with minimum degree at least $n/2$ has a planar subgraph 
with $2n-4$ edges. So it is natural to ask whether there are values $\gamma<1/2$ such 
that $C(\gamma)=4$. We answer this in the affirmative as we determine the optimal 
value of $C(\gamma)$ for all $0<\gamma<1/2$. 

\begin{theorem}\label{thm:main} 
 For every $\gamma\in(0,1/2)$ there exists $n_\gamma$ such that 
 $\pl(n,\gamma n) = 2n-4k$ for every $n\ge n_\gamma$, where 
 $k\in\NATS$ is the unique integer such that $k\le 1/(2\gamma)< k+1$. 
 Hence, $C(\gamma)=4\lfloor1/(2\gamma)\rfloor$ for $n\ge n_\gamma$.
\end{theorem}

Note that the constants are best possible for the given minimum degree 
condition: the graph consisting of $k$ disjoint copies of $K_{t,t}$ has 
$2kt$ vertices, is $t$-regular, and has no planar subgraph with more 
than $4kt-4k$ edges because $K_{t,t}$ has no planar subgraph with more 
than $4t-4$ edges.

\section{Tools and lemmas}

Our main tools in the proof are variants of the Regularity Lemma~\cite{Szem78} 
and the Blow-up Lemma~\cite{KSS97}. In order to formulate the versions 
that we will use, we first introduce some terminology.
 
Let $G=(V, E)$ be a graph and let $\eps,d\in(0,1]$. For disjoint nonempty sets 
$U,W\subset V$, we denote by $e(U, W)$ the number of edges between $U$
and $W$, and define the \emph{density} of the pair $(U,W)$ as 
$d(U,W):=e(U,W)/|U||W|$. A pair $(U,W)$ is \emph{$\eps$-regular} if
$$|d(U',W')-d(U,W)|\le\eps$$ for all $U'\subset U$ and $W'\subset W$ with
$|U'|\ge\eps|U|$ and $|W'|\ge\eps|W|$. If the pair $(U,W)$ is $\eps$-regular 
and has density at least $d$, then we say that  $(U,W)$ is \emph{$(\eps,d)$-regular}.

An \emph{$\eps$-regular partition} of~$G=(V, E)$
is a partition $V_0\dcup V_1\dcup\dots\dcup V_r$ of $V$ with $|V_0|\le\eps |V|$,
$|V_i|=|V_j|$ for all $i,j\in[r]:=\{1,\ldots,r\}$, and such that, for all but at most 
$\eps r^2$ pairs $(i,j)\in[r]^2$, the pair $(V_i,V_j)$ is $\eps$-regular.

We say that an $\eps$-regular partition $V_0\dcup V_1\dcup\dots\dcup V_r$ of a 
graph $G$ is an \emph{$(\eps,d)$-regular partition} if the following is true. For every
$i\in [r]$ and every vertex $v\in V_i$, there are at most $(\eps+d)n$ edges
incident to $v$ which are not contained in $(\eps,d)$-regular pairs of the
partition.

Given an $(\eps,d)$-regular partition $V_0\dcup V_1\dcup\dots\dcup V_r$ of a
graph $G$, we define a graph $R$, called the \emph{reduced graph} of the
partition of $G$, where $R=(V(R),E(R))$ has $V(R)=\{V_1,\ldots,V_r\}$ and
$V_iV_j\in E(R)$ whenever $(V_i,V_j)$ is an $(\eps,d)$-regular pair. We will
usually omit the partition, and simply say that $G$ has
\emph{$(\eps,d)$-reduced graph}~$R$. We call the partition classes $V_i$ with
$i\in[r]$ \emph{clusters} of $G$. Observe that our definition of the reduced
graph~$R$ implies that, for $T\subset V(R)$, we can, for example, refer to the set
$\bigcup T$, which is a subset of $V(G)$.

In our proof, we require the minimum degree form of the Regularity Lemma.
\begin{lemma}[Regularity Lemma, minimum degree form]\label{lem:plex:reg}
 For all positive $\eps$, $d$ and $\gamma$ with $0<\eps<d<\gamma< 1$ 
 there is $r_1$ such that every graph $G$ on $n>r_1$ vertices with minimum 
 degree $\delta(G)\ge\gamma n$ has an $(\eps,d)$-reduced graph~$R$ on $r$ 
 vertices such that $r\le r_1$ and $\delta(R)\ge(\gamma-d-\eps)r$.
\end{lemma}
Lemma \ref{lem:plex:reg} is an easy consequence of the original Regularity Lemma of 
Szemer\'edi~\cite{Szem78}. Its proof can be found, for example, in~\cite[Proposition~9]{KuOsTa05}.

Now we outline our proof strategy for Theorem~\ref{thm:main}. First, we apply 
Lemma \ref{lem:plex:reg} to a~given $n$-vertex graph $G$ with minimum degree at 
least $\gamma n$ and obtain the reduced $r$-vertex graph $R$ whose minimum 
degree is almost as large as $\gamma r$. Then we need to distinguish two cases.

If any component of $R$ has less than $2\delta(R)$ vertices, then it contains a
triangle. Using this triangle we will find a small triangulation $T$ (that is, a plane 
graph whose every face is a triangle) in $G$, and Theorem~\ref{thm:KOTquad} will 
guarantee a subgraph $S$ of the rest of the graph $G-V(T)$ such that the disjoint 
union of $S$ and $T$ has at least $2n$ edges. 

It follows that each component of $R$ has at least $2\delta(R)>r/(k+1)$ vertices, 
and thus $R$ has at most $k$ components. These components correspond to $k$
well-connected subgraphs of $G$ and cover almost all vertices of $G$. In each 
subgraph we will find a quadrangulation (a plane graph whose every face has four 
edges) which has a certain `accepting' property that allows the few remaining vertices 
to be inserted. We conclude that there is a collection of at most $k$ vertex-disjoint 
quadrangulations covering all the vertices of~$G$. Since every quadrangulation on 
$m$ vertices has $2m-4$ edges, the theorem follows.

As one can see from the above outline, our argument divides into two cases, 
depending on whether the reduced graph $R$ has a small component or not. In each 
case we shall need some embedding results, which we now describe in detail.

When the reduced graph $R$ does have a small component, we will need the 
following embedding result, an easy case of the Counting Lemma (see, for example, 
Theorem 2.1 in~\cite{KS96}). 

\begin{lemma}\label{lem:count} For each $d>0$ and $s\in\NATS$ there exist
 $\eps>0$ and $m_0$ such that whenever $m\ge m_0$ the following holds. Let
 $U,V,W$ be three pairwise disjoint vertex sets each of size $m$. Suppose that
 each pair forms an $(\eps,d)$-regular pair in a graph $G$. Then $G$ contains
 every $3$-partite triangulation on $s$ vertices.
\end{lemma}

In the case that $R$ has no small components, we will construct
quadrangulations. For this we shall use a version of the Blow-up Lemma. 
In order to state this result, we need a further definition. A pair of disjoint sets of 
vertices $U$ and $W$ in a graph $G$ is called \emph{$(\eps,\delta)$-super-regular} 
if it is $\eps$-regular, each vertex $u\in U$ has at least $\delta |W|$ neighbours in 
$W$, and each $w\in W$ has at least $\delta |U|$ neighbours in~$U$. 

The original version of the Blow-up Lemma, due to Koml\'os, S\'ark\"ozy and 
Szemer\'edi~\cite{KSS97}, showed that, for the purposes of embedding graphs 
of bounded degree, super-regular pairs behave like complete bipartite graphs. In our 
proof, we will need to embed (planar) graphs with growing degrees, which is 
generally a very difficult problem. Fortunately for us, planar graphs are examples of 
arrangeable graphs, for which a suitable extension of the Blow-up Lemma~\cite{KSS97} 
has recently been proven by B\"ottcher, Kohayakawa, Taraz and W\"urfl. 

\begin{definition}[$a$-arrangeable]  \label{def:plex:arrangeable}
Let $a$ be an integer. An $n$-vertex graph is called $a$-arrangeable if its vertices
can be ordered as $(x_1,\dots ,x_n)$ in such a way that
\[
  \left|N\Big(N(x_i)\cap\{x_{i+1}, x_{i+2},\dots, x_n\}\Big) \cap  \{x_1, x_2,\dots, x_i \}\right| \le a
\]
for each $1 \le i \le n$. 

Here, for a set of vertices $S$, we denote by $N(S)$ the set of 
those vertices not in $S$ that are adjacent to some vertex in $S$.
\end{definition}

Chen and Schelp showed that planar graphs are $761$-arrangeable \cite{CheShe93}; 
Kierstead and Trotter \cite{KieTro93} improved this to $10$-arrangeable. Thus, 
the following theorem of B\"ottcher, Kohayakawa, Taraz and W\"urfl~\cite{BKTW_blowup} 
can be used to embed planar graphs whose maximum degree is not too large.

\begin{theorem}[Arrangeable Blow-up Lemma] \label{thm:Blow-up:arr:plex}
For all $a,\Delta_R,\kappa \in \NATS$ and for all $\delta>0$ there
exists $\eps>0$ such that for every integer $r$ there is $n_0$
such that the following is true for every $n_1,\dots,n_r$ with $n_0\le n
= \sum n_i$ and $n_i \le \kappa\cdot n_j$ for all $i,j\in[r]$. 

Let $R$ be a graph of order $r$ with $\Delta(R)<\Delta_R$.
Assume that we are given a graph $G$ with a partition 
$V(G)=V_1\dcup\dots\dcup V_r$ and a graph $H$ with a partition 
$V(H)=X_1\dcup\dots\dcup X_r$ with $|V_i|=|X_i|=n_i$ such that 
$(V_i,V_j)$ is an $(\eps,\delta)$-super-regular pair for every 
$ij\in E(R)$ and such that all edges of $H$ run between sets $X_i,X_j$ 
for which $ij\in E(R)$. Further assume that $H$ is $a$-arrangeable and 
has $\Delta(H)\le \sqrt{n}/\log n$. Then there exists an embedding 
$\phi: V(H) \to V(G)$ of $H$ to $G$ such that $\phi(X_i)=V_i$.
\end{theorem}

%A spanning tree in a component of $R$ corresponds to a large well 
%connected subgraph of $G$. We will embed planar graphs into
%spanning trees of components of $R$. However we need spanning trees whose
%maximum degree is bounded: this is the purpose of the following lemma.

We will embed planar graphs into large well connected subgraphs of $G$. 
These subgraphs will correspond to spanning trees of components of $R$. 
However, to be able to use Theorem \ref{thm:Blow-up:arr:plex}, we shall need 
spanning trees whose maximum degree is bounded. This is the purpose of the 
following lemma.

\begin{lemma}\label{lem:spanbound}
 Given $k\in \NATS$, let $R$ be a connected graph with minimum degree at
 least $v(R)/(2k)$. Then $R$ has a spanning tree with maximum degree $8k$.
\end{lemma}
\begin{proof}
 We define the \emph{score} of a spanning tree $T$ of $R$ to be the sum of the
 squares of the degrees of vertices in $T$. Let $T$ be a spanning tree of $R$
 with minimum score. Observe that $T$ has less than $v(R)/(4k)$ vertices of
 degree $8k$, since the sum of the vertex degrees of $T$ is $2v(R)-2$.

 Suppose that there is a vertex $u$ of $T$ whose degree in $T$ exceeds $8k$.
 Observe that the removal of $u$ from $T$ disconnects $T$ into more than $8k$
 components, one of which, $C$, has less than $v(R)/(8k)$ vertices. Let $v$ be
 the neighbour of $u$ which is in $C$. Now $v$ has at
 least $v(R)/(2k)$ neighbours in $R$, of which less than $v(R)/(8k)$ are in $C$
 and a further less than $v(R)/(4k)$ are of degree at least $8k$. It follows
 that $v$ has a neighbour $u'$ in $R$ which is not in $C$ and whose degree is
 less than $8k$. Let $T'$ be obtained from $T$ by deleting $uv$ and inserting
 $u'v$. Then $T'$ is still a spanning tree of $v$. Each vertex of $T'$ has the
 same degree as in $T$ except for $u$ and $u'$, which have respectively lost and
 gained one neighbour. It follows that the score of $T'$ is smaller than
 that of $T$, which by contradiction completes the proof.
\end{proof}

Finally, we need to specify which planar graphs we will embed into our spanning
trees.

Let $H=(V,E)$ be a plane graph. We say that a $k$ element subset 
$V'\subseteq V$ forms a \emph{bag of order $k$} in $H$ if there is
$\{x_1,x_2\}\subseteq V\setminus V'$ such that $V'\cup\{x_1,x_2\}$
induces a copy of $K_{2,k}$ in $H$ and all inner faces of
$H[V'\cup\{x_1,x_2\}]$ are also faces of $H$. We call the vertices of $V'$
which are not in the outer face the \emph{interior vertices} of the bag.

\begin{figure}[ht]
\begin{center}
\begin{tikzpicture}[inner sep=1.75pt]

%%%%%%%%%%%%%%%%%%%%%%%%%%%%%%%%%%%%%%%%%%%%%

  \node[label=270:$x_1$] (0) at (0,0) {};
  \node (1) at (-2.2,1) {};
  \node[label=180:$u$] (2) at (-1.2,1) {};
  \node (3) at (-.4,1) {};
  \node (4) at (.4,1) {};
  \node[label=0:$u'$] (5) at (1.2,1) {};
  \node (6) at (2.2,1) {};
  \node[label=90:$x_2$] (7) at (0,2) {};

  \node[label=270:$x_1$] (8) at (6,0) {};
  \node (9) at (3.8,1) {};
  \node (10) at (4.5,1) {};
  \node (11) at (5.1,1) {};
  \node[label=180:$u$] (12) at (5.8,1) {};
  \node[label=0:$u'$] (13) at (7.2,1) {};
  \node (14) at (8.2,1) {};
  \node[label=90:$x_2$] (15) at (6,2) {};

  \node (16) at (6.35,.5) {};
  \node (17) at (6.35,.7) {};
  \node (18) at (6.35,.9) {};
  \node (19) at (6.35,1.1) {};
  \node (20) at (6.35,1.3) {};
  \node (21) at (6.35,1.5) {};

  \draw[black,thick] 
             (1) -- (0) -- (2) -- (7) -- (3) -- (0) -- (4) -- (7) -- (5) -- (0) -- (6) -- (7) -- (1)
             (9) -- (8) -- (10) -- (15) -- (11) -- (8) -- (12) -- (15) -- (13) -- (8) -- (14) -- (15) -- (9);

  \draw[black]
            (12) -- (16) -- (13) -- (17) -- (12) -- (18) -- (13) -- (19) -- (12) -- (20) -- (13) -- (21) -- (12);

%%%%%%%%%%%%%%%%%%%%%%%%%%%%%%%%%%%%%%%%%%%%%%%%%%%%%

  \foreach \i in {0,...,15}
  {
    \fill[black!80] (\i) circle (3pt);
  }

  \foreach \i in {16,...,21}
  {
    \fill[black!80] (\i) circle (2pt);
  }

\end{tikzpicture}
\caption{A bag of order $k=6$; the same bag reordered and with an insertion of $\ell=6$ vertices.}
\label{fig:sub:bags}
\end{center}
\end{figure}
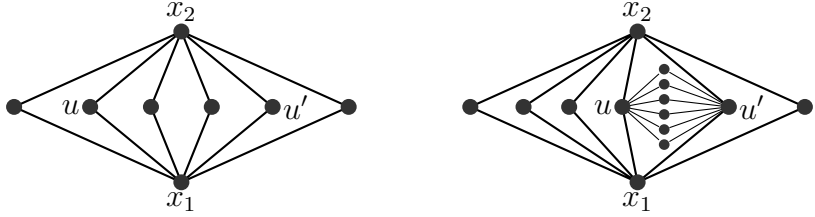

Observe that we can reorder the interior vertices of a bag without affecting
the planarity of $H$. A bag is thus a very convenient structure into which one
can put further vertices: if some vertex $v$ (not in $H$) is adjacent in a
supergraph $G$ of $H$ to any two interior vertices $u,u'$ of a~bag, then we can
redraw $H$ such that $u$ and $u'$ are consecutive in the bag, and insert $v$ and
$uv,u'v$ to obtain $H'$. If $H$ is a quadrangulation, then $H'$ is still a
quadrangulation contained in $G$. Furthermore, if $v_1,\ldots,v_\ell\notin V(H)$
are all adjacent to $u$ and $u'$ in $G$, then we can insert all these vertices
and edges to $u$ and $u'$, and still obtain a subgraph $H''$ of $G$ which is a
quadrangulation. Furthermore, $v_1,\ldots,v_\ell$ then form a~bag of order $\ell$ in
$H''$. This will be particularly useful in the proof of the following lemma.

\begin{lemma} \label{lem:blow-up-of-tree}
 Let $T$ be a tree of order $r\ge 2$, $n\ge (16r)^3$, and let $G$ be an
 $n$-vertex graph with the partition $V_1\dcup\dots\dcup V_r$ of its
 vertex set such that $|V_i|\le 2|V_j|$ for all $i\neq j$ and $G[V_i,V_j]$
 is a complete bipartite graph whenever $ij\in E(T)$. 

 Then $G$ contains a plane quadrangulation $H$ with maximum degree
 $\Delta(H)\le n^{1/3}+2$ as a spanning subgraph. Furthermore, all but
 at most $9 n^{2/3}$ vertices of $H$ are contained in a collection of 
 pairwise disjoint bags each of order in the interval 
 $[n^{1/3}/2, n^{1/3}]$.
\end{lemma}

\begin{proof}
We first prove that $G$ has a quadrangulation $H$ with 
$\Delta(H)\le n^{1/3}+2$ by induction on $r$. So assume that $r=2$ and 
$G$ is a bipartite graph with partite sets $V_1,V_2$. We partition $V_i$
into a minimum number of sets $W_{i,j}$ with sizes $|W_{i,j}|\le n^{1/3}$ as equal as 
possible. The plane graph $H$ is constructed as follows. Take 
$x_1,x_2\in W_{1,1}$ and all of $W_{2,1}$ and embed the graph induced by
these vertices into the plane. Let $y_1,y_2\in W_{2,1}$ lie in the same 
face and embed $W_{1,1}\setminus \{x_1,x_2\}$ into this face connecting 
each vertex to $y_1$ and $y_2$. We continue greedily embedding sets 
$W_{i,j}$ into faces with two vertices of degree 2 from $V_{3-i}$ and 
adding all edges in between. This process does not stop before all 
vertices of $G$ have been embedded into the plane. The resulting graph 
$H$ is a quadrangulation with $\Delta(H)\le n^{1/3}+2$.

Now assume that $r>2$ and $1r\in E(T)$. Further assume that we have 
embedded $V_1\dcup \dots\dcup V_{r-1}$ this way and obtained a 
quadrangulation $H'$ on $V\setminus V_r$. We extend $H'$ to a 
quadrangulation $H$ on $V$ as follows. Again partition $V_r$ into a minimum number of sets $W_{r,i}$ with sizes $|W_{r,i}|\le n^{1/3}$ as equal as possible. For each $i$ sequentially, pick a pair of 
vertices $u, u'$ in $V_1$ that have degree 2 and lie in the same face in $H'$, embed all vertices from $W_{r,i}$ into this face, and connect these to $u, u'$.
Since $|V_1|\ge n/(2r-1)> 8n^{2/3}$, we do not run out of pairs in~$V_1$.

It remains to show that most vertices lie in a collection of large disjoint bags. 
Recall that the planarity of $H$ is preserved if we reorder the embedding in 
such a way that all vertices in $W_{i,j}$ of degree~2 in $H$ form a bag. 
Since there are at most $2n^{2/3}$ many sets $W_{i,j}$, all but at most 
$4n^{2/3}$ many vertices lie in pairwise disjoint bags. Some of these bags 
might be small, i.e., they might have order less than $\tfrac{1}{2}n^{1/3}$. 
Assume that the bag in $W_{i,j}$ is small. Note that 
$|W_{i,j}|\ge \tfrac{9}{10}n^{1/3}$ by construction. Thus at least 
$\tfrac15 n^{1/3}$ pairs from $W_{i,j}$ have been used to embed other 
sets $W_{i',j'}$. But there are at most $2n^{2/3}$ many sets $W_{i,j}$. 
Hence, at most $10n^{1/3}$ bags are small. Consequently, all but at most 
$4n^{2/3}+10n^{1/3}\cdot \tfrac{1}{2}n^{1/3}\le 9 n^{2/3}$ vertices lie 
in disjoint bags of size at least~$\tfrac12 n^{1/3}$.
% Choose another two vertices of $W_{1,1}$ that lie in the same face and have degree 2 and embed all vertices of $W_{2,2}$ into this face again connecting all new vertices to the two previously embedded ones. We continue embedding $W_{i,j}$ 
\end{proof}

\section{Proof of Theorem~\ref{thm:main}}

Given $\gamma>0$, let $k\in \NATS$ be such that $k\le 1/(2\gamma)<k+1$. 
We set $\beta=\gamma-1/(2(k+1))$, $\delta=\beta/8$, $d=\beta/4$ and 
$s=C(\beta)+6$, where $C(\beta)$ is the constant returned by Theorem~\ref{thm:KOTquad}. 
Next we choose $\eps$ such that $2\eps$ is sufficiently small to apply Theorem~\ref{thm:Blow-up:arr:plex} 
with $a=10$, $\Delta_R=8(k+1)+1$, $\kappa=2$ and $\delta$ as given.
We further insist that
\[
 \eps\le \frac{\beta}{10^5k^4(8(k+1)+2)}\,.
\]

Let $r_1$ be the parameter returned by Lemma~\ref{lem:plex:reg} for $d$ and $\eps$
as chosen, and $m_0$ that returned by Lemma~\ref{lem:count}. Let $n_0\ge \max\big\{(16r_1)^3,6(k+1)s,m_0r_1\big\}$ be sufficiently large so that
Theorem~\ref{thm:Blow-up:arr:plex} applies with any $r\le r_1$.

Suppose that $n\ge (4k)\cdot n_0$ and let $G$ be an $n$-vertex graph with minimum degree 
$\delta(G)\ge\big(\tfrac{1}{2(k+1)}+\beta\big)n$.
By Lemma~\ref{lem:plex:reg} there is an $(\eps, d)$-regular partition 
$V(G)=V_0\dcup\ldots\dcup V_r$ with $r\le r_1$ such that the corresponding
reduced graph $R$ satisfies $\delta(R)\ge \big(\tfrac{1}{2(k+1)}+\beta/2\big)r$.
We distinguish two cases.

\underline{Case 1:} $R$ has a component with less than $2\delta(R)$ vertices.
In this case $R$ contains a triangle. It follows by Lemma~\ref{lem:count} that
$G$ contains a triangulation $T$ on $s$ vertices, which has $3s-6$ edges. The
graph $G-V(T)$ has minimum degree at least
$\big(\tfrac{1}{2(k+1)}+\beta\big)n-s\ge\beta n$, where the last inequality is by
our choice of $n_0$. Therefore, by Theorem~\ref{thm:KOTquad}, $G-V(T)$ 
contains a planar subgraph $S$ with at least $2(n-s)-C(\beta)$ edges. Then 
$G$ contains the disjoint union of $S$ and $T$, which is planar and has 
at least $2n-2s-C(\beta)+3s-6=2n$ edges (by choice of $s$) as required.

\underline{Case 2:} Every component of $R$ has at least $2\delta(R)>r/(k+1)$
vertices. It follows that $R$ has $c\le k$ components. We will show that we
can cover $G$ with $c$ vertex-disjoint quadrangulations, which implies that $G$
contains a planar subgraph with at least $2n-4c\ge 2n-4k$ edges as required.

Let $C$ be a component of $R$, and $T$ be its spanning tree with maximum 
degree $8(k+1)$ guaranteed by Lemma~\ref{lem:spanbound}. Let $V_i$ be any
cluster of $C$ and $ij\in T$. Observe that, by the $(\eps, d)$-regularity of
$(V_i, V_j)$, at most $\eps |V_i|$ vertices do \emph{not} have at least 
$(d-\eps)|V_j|$ neighbours in $V_j$.  It follows that we can remove from each
cluster $V_i$ at most $8(k+1)\eps|V_i|$ vertices and obtain a set $V'_i$ whose 
every vertex has at least $(d-\eps)|V_j|$ neighbours in each $V_j$ such that 
$ij\in T$. Since $(8(k+1)+1)\eps\le\beta/8$ and $d-\delta=\beta/8$, if $ij\in T$, 
then each vertex in $V_i'$ has at least $\delta |V'_j|$ neighbours in $V'_j$. 
Moreover, since  $8(k+1)\eps\le 1/2$, for each $i$ we have $|V'_i|\ge |V_i|/2$.
Consequently,  the pair $(V'_i,V'_j)$ is $2\eps$-regular (see~\cite[Fact 1.5]{KS96}) 
and, since $|V_i|=|V_j|$, we also have $|V'_i|\le 2|V'_j|$ for each $i,j$. It follows 
that each edge $ij$ of $T$ corresponds to a $(2\eps,\delta)$-super-regular pair 
$(V'_i,V'_j)$, and the cluster sizes are not too unbalanced, as required for 
Theorem~\ref{thm:Blow-up:arr:plex}.

Let $G'$ be a graph whose vertex set is the union of the sets $V'_i$, $i\in C$ and 
whose edges are all edges between $V'_i$ and $V'_j$ whenever $ij\in T$. Note that
$G'$ has $n_C$ vertices, where $n_C$ satisfies
$$
 n\ge n_C\geq 2\delta(R)\big(1-8(k+1)\eps\big)\frac{(1-\eps)n}{r}\ge \frac{n}{2k}\ge n_0.
$$
By Lemma~\ref{lem:blow-up-of-tree}, $G'$ contains a plane quadrangulation $H$ in
which the maximum degree is at most ${n_C}^{1/3}+2\le n^{1/3}+2$, and in which 
at most $9n_C^{2/3}$ vertices are not contained in bags of order between 
$\tfrac12 {n_C}^{1/3}\geq (n/16k)^{1/3}$ and $n_C^{1/3}\le n^{1/3}$. By 
Theorem~\ref{thm:Blow-up:arr:plex}, $H$ can be embedded into the subgraph 
of $G$ induced on $\bigcup C$.

Repeating this for each component we obtain $c$ vertex disjoint
quadrangulations $H_1,\ldots,H_c$ in $G$, together with a collection
$B_1,\ldots,B_\ell$ of pairwise disjoint bags of order at least $(n/16k)^{1/3}$
covering all but at most $9\sum_{C}n_C^{2/3}\le 9n^{2/3}k$ vertices of 
$\bigcup_{i\in[r]}V'_i$. In particular, we have that
$\tfrac12n^{2/3}<\ell\le (16k)^{1/3}n^{2/3}$.

Let $L$ be the set of vertices in none of the quadrangulations. Observe 
that every vertex in $L$ is either in $V_0$ or in $V_i\setminus V'_i$ for 
some $i$; therefore, it follows that $|L|\le(8(k+1)+1)\eps n$. We say that
a bag $B_i$ is \emph{good} for $u\in L$ if $u$ has at least $n^{1/3}/(32k^2)$ 
neighbours in $B_i$. Since $(8(k+1)+1)\eps n+9n^{2/3}k\le\beta n$, each 
vertex $u\in L$ has at least $n/(2(k+1))$ neighbours contained in 
$B_1\cup\cdots\cup B_\ell$. Of these at least 
$n/(2(k+1))-\ell\cdot n^{1/3}/(32k^2)\ge n/(6k)$ lie in bags that are 
good for $u$. Hence, at least $n^{2/3}/(6k)\ge\ell/(24k^2)$ of the bags 
$B_i$ are good for $u$. We now assign vertices $L_i$ of $L$ to each bag 
$B_i$ sequentially as follows. From the collection of unassigned vertices 
of $L$ for which $B_i$ is good, we assign $L_i$ to be any $n^{1/3}/(128k^2)$ 
of them if this is possible, and all of them if not. Suppose that after 
carrying out this procedure there is a vertex $u$ of $L$ which is not in 
any $L_i$. Then it must be the case that for each $B_i$ good for $u$, we have
$|L_i|=n^{1/3}/(128k^2)$. But there are at least $\ell/(24k^2)$ such $B_i$, and
$\ell/(24k^2)\cdot n^{1/3}/(128k^2)>(8(k+1)+1)\eps n\ge|L|$, (where the first
inequality is by choice of $\eps$) which is a contradiction.

We then work as follows. For each bag $B_i$, we reorder the interior vertices
of $B_i$ such that the first vertex of $L_i$ is adjacent to the first and second
interior vertices of $B_i$, the second vertex of $L_i$ to the third and fourth,
and so on. Because each vertex of $L_i$ has at least $n^{1/3}/(32k^2)$
neighbours in $B_i$, and $|L_i|\le n^{1/3}/(128k^2)$, this is possible. We 
now insert, for each $j$, the $j$th vertex of $L_i$ into the interior face of $B_i$
containing the $(2j-1)$st and $2j$th interior vertices, and add the edges to
those two vertices. Let the plane graphs so constructed be $H'_1,\ldots,H'_c$.
By construction, these graphs are vertex disjoint and cover $G$, and since
$H_i$ was a quadrangulation, so $H'_i$ is also a quadrangulation for each $i$.
The disjoint union of $H'_1,\ldots,H'_c$ is then a planar subgraph of $G$ with
$2n-4c\ge 2n-4k$ edges, as required.

\section{Concluding Remarks}

There remain several open questions on planar graphs. In particular, it is
possible that in Theorem~\ref{thm:main} the constant $n_\gamma$ can be taken to
be an absolute constant provided $\gamma\gg n^{-1/2}$. Note that this is a
natural lower bound since there are bipartite graphs without $4$-cycles of
minimum degree $\Theta(n^{1/2})$. Another possibility would be to investigate
the behaviour of the planarity function $\pl(n,\gamma n)$ for
$\gamma\in(1/2,2/3]$ in more detail. Finally, one could ask these questions if
the constraint imposed is that of edge density rather than minimum degree.

More generally, one could replace `planar graphs' by some other
property --- topologically defined, or by forbidden minors, for example.

\section*{Acknowledgement}

The third author would like to thank the London School of Economics for their hospitality while this work was being completed.

%%%% BIB %%%%%%%%%%%%%%%%%%%%%%%%%%%%%%%
 
\bibliographystyle{amsplain} 
\bibliography{PlanarSub2n4}

\end{document}